\documentclass[11pt,dvipdfmx,a4paper]{article}
\usepackage[margin=15mm]{geometry}
\usepackage{amssymb,amsthm,enumerate,color}
\usepackage{amsmath,fancybox,amsfonts,ascmac,bm,ulem}
\usepackage{graphicx, tikz}
\usepackage{comment}
\numberwithin{equation}{section}

\usepackage[colorlinks=true,linkcolor=blue,citecolor=blue]{hyperref}
\theoremstyle{definition}
\newtheorem{thm}{Theorem}[section]

\newtheorem{lem}[thm]{Lemma}
\newtheorem{prop}[thm]{Proposition}

\newtheorem{rem}[thm]{Remark}

\newcommand{\R}{\mathbb{R}}   
   
\newcommand{\N}{\mathbb{N}}   
\newcommand{\Z}{\mathbb{Z}}

\newcommand{\ba}{{\bm{a}}}

\newcommand{\bc}{{\bm{c}}}
\newcommand{\bx}{{\bm{x}}}
\newcommand{\by}{{\bm{y}}}

\allowdisplaybreaks

\begin{document}

\title{
\vspace{-0.75cm}
\Large{\bf Higher-order Asymptotic Expansion with Error Estimate for the Multidimensional Laplace-type Integral under Perturbations}
}
\author{Ikki Fukuda, Yoshiki Kagaya and Yuki Ueda}
\date{}

\maketitle


\vspace{-1.25cm}
\abstract{
We consider the asymptotic behavior of the multidimensional Laplace-type integral with a perturbed phase function. 
Under suitable assumptions, we derive a higher-order asymptotic expansion with an error estimate, generalizing some previous results including Laplace's method. 
The key points of the proof are a precise asymptotic analysis based on a lot of detailed Taylor expansions, and a careful consideration of the effects of the perturbations on the Hessian matrix of the phase function. 
}


\section{Introduction}

\indent

We consider the asymptotic behavior of the following integral as $n\to \infty$:
\begin{align}\label{eq:Laplace_integral}
I_{n} = \int_{\Omega} e^{nh_{n}(\bx)} g(\bx) d\bx, \quad n\in \mathbb{N},
\end{align}
where $\Omega$ is a bounded domain in $\R^d$, $d \in \mathbb{N}$ and $g : \overline{\Omega} \to \R$ are sufficiently smooth. 
On the other hand, $h_{n} : \overline{\Omega} \to \R$ are given by the form $h_{n}(\bx)=h(\bx)+\varepsilon_{n}\sigma(\bx)$ for all $n\in \mathbb{N}$, where $\varepsilon_{n} >0$ satisfies $\varepsilon_{n} \to 0$ as $n\to \infty$, and $h, \sigma : \overline{\Omega} \to \R$ are also sufficiently smooth (we will explain the detailed assumptions on the decay order of $\varepsilon_{n}$ and on the smoothness of the above functions later). 
This integral \eqref{eq:Laplace_integral} is called the Laplace-type integral, which has several applications in various fields of science and engineering as well as mathematics, and the literature on this integral is quite extensive. 
In particular, the Laplace-type integral \eqref{eq:Laplace_integral} often appears in the probability theory and statistics (see, e.g. \cite{B23, B86, DZ98, KTK90, M82, M08, DZ03, W19, WU24} and also references therein). 
For example, we note that the asymptotic analysis for \eqref{eq:Laplace_integral} can be applied to the large deviation principle (cf. \cite{DZ98,DZ03}) and Bayesian statistics (cf. \cite{B23, KTK90, M08}). 
Moreover, it is also applied to the study of speech recognition \cite{FDAK01} and signal processing \cite{MQG11}.  
Furthermore, it helps reliability analysis method combined with artificial neural network too \cite{JW22}. 
In addition, it also has applications in the field of random chaos \cite{KPH15}. 
Although there is a huge amount of literature on the Laplace-type integral other than those mentioned here, 
there are not a few things that are not mathematically rigorous. 
Therefore, we believe that providing some theoretical results of \eqref{eq:Laplace_integral} might be useful in not only pure mathematics, but also in many areas of science and engineering. 
In research aimed at generalizing the Laplace-type integral, the one-dimensional case $d=1$ is frequently examined (cf. \cite{FK24, N20, O94, W19, WU24} and also references therein). 
However, we focus on the multidimensional case in this paper, since it is important in view of some applications (e.g. \cite{B23, KTK90, K20, L19, M08, MQG11}). 
The purpose of this paper is to derive a higher-order asymptotic expansion with an error estimate for the multidimensional Laplace-type integral \eqref{eq:Laplace_integral}. 
In particular, we deal with the case where the phase function includes the perturbations $\varepsilon_{n}\sigma(\bx)$, and analyze the effect of the perturbations on the asymptotic behavior of $I_{n}$ as $n\to \infty$.

Before presenting our main result, let us introduce some known results related to the theoretical analysis for \eqref{eq:Laplace_integral}. 
For the asymptotic behavior of $I_{n}$ as $n\to \infty$, a famous technique called Laplace's method or Laplace's approximation introduced by Laplace~\cite{L12} is well known. It has been developed in various forms over many years of study. 
For example, we can refer to \cite{FK24, KTK90, K10, K20, L19, M08, N20, O94, W19, WU24} and also references therein. 
First, we shall explain the classical result of an asymptotic analysis for \eqref{eq:Laplace_integral} with $\sigma(\bx)\equiv 0$. 
Assume that $h : \overline{\Omega} \to \R$ has a maximum only at $\bx = \bc \in \Omega$, and is twice continuously differentiable function around $\bx=\bc$ satisfying $\det D^{2}h(\bc)\neq 0$, where $D^{2}h(\bc)$ is the Hessian matrix of $h$ at $\bx = \bc$. 
Then, if $g(\bc)\neq0$, $I_{n}$ satisfies the following asymptotic formula:
\begin{equation}\label{thm-Laplace-1}
I_{n} \sim e^{nh(\bc)}n^{-\frac{d}{2}}g(\bc)\sqrt{\frac{(2\pi)^d}{|\det D^{2}h(\bc)|}}
\end{equation}
as $n\to \infty$. The notation ``$\sim$'' is used to mean that the quotient of the left hand side by the right hand side converges 1 as $n\to \infty$. 
As we mentioned in the above, some related results to this formula (including the case of $\sigma(\bx) \neq0$) have already been obtained by many mathematicians. 
Indeed, as a classic reference, we can refer to Olver's textbook \cite{O94}. 
Moreover, for the multidimensional case, let us introduce the textbooks by Bleistein--Handelsman \cite{BH75} and Wong \cite{W89}. 
Furthermore, some generalized asymptotic formulas for $I_{n}$ with $\sigma(\bx)\equiv 0$ can be found in Kirwin \cite{K10} and Katsevich \cite{KP} for example (see, also \cite{KTK90, M08}), under different assumptions and methods from ours below. 
Furthermore, \L{}api\'{n}ski \cite{L19} and Kolokoltsov \cite{K20} obtained some asymptotic formulas for the Laplace-type integral when the phase function depends on not only $\bx \in \Omega$ but also $n\in \mathbb{N}$ like our setting. 
Also, we can see a lot of generalizations of \eqref{thm-Laplace-1} for the one-dimensional case $d=1$, e.g. \cite{FK24, N20, O94, W19, WU24}. 
Especially, as a result of its most generalized form (up to the authors knowledge), let us refer to the recent paper by Nemes~\cite{N20}. 

Finally, we shall explain some motivations of our study. 
Although the above result \eqref{thm-Laplace-1} is useful, there are some problems that could easily arise. 
For example, when the maximum point of $h(\bx)$ and the zero point of $g(\bx)$ overlap, i.e., when $g(\bc)=0$, this formula does not make sense, 
because the asymptotic profile vanishes identically. 
Therefore, for the more general $I_{n}$ in which such situation occurs, we need to construct a new asymptotic formula. 
Moreover, from an application point of view, an error between $I_{n}$ and the asymptotic profile should also be investigated. 
In the previous studies mentioned above, some formulas that overcome these problems have been obtained in \cite{FK24, K10, N20, O94}. 
However, in order to accommodate more general situations, we aim to derive a new formula that further develops their results in some sense. 
Another purpose is to derive a better higher-order asymptotic formula for the multidimensional Laplace-type integral in which the phase function depends on both $\bx\in \Omega$ and $n\in \mathbb{N}$, since these integral is important from the application point of view.
Actually, such formulas have been derived in \cite{K20, L19, N20, O94}. As an application to probability theory, the rates of convergence of moment generating functions associated with the weak law of large numbers and the central limit theorem have been established in \cite{K20, L19}. 
In particular, the phase function of the form $h_{n}(\bx)=h(\bx)+\varepsilon_{n}\sigma(\bx)$ with $\varepsilon_{n} >0$ satisfying $\varepsilon_{n} \to 0$ as $n\to \infty$, has been treated by \L{}api\'{n}ski \cite{L19}, and our study is inspired by his work. 
Based on the above considerations, in our study, we analyzed the multidimensional Laplace-type integral with a perturbed phase function \eqref{eq:Laplace_integral} and succeeded in deriving a result on the higher-order asymptotic expansion for $I_{n}$, together with an error estimate, accomplishing the above objectives.
We note that the method used in this paper requires only basic calculus and elementary linear algebra. 
Therefore, one of the selling points of this paper is that it provides an analysis that is accessible not only to mathematicians but also to readers who are not specialized in mathematics. 

\medskip
\par\noindent
\textbf{\bf{Notations}} 
\begin{itemize}
\item 
In this paper, $\alpha = (\alpha_1,\dots, \alpha_d)\in \Z_{\ge 0}^d$ means a multi-index. 
Then, we define $\alpha! := \alpha_1 ! \cdots \alpha_d !$ and $|\alpha| := \alpha_1+\cdots+ \alpha_d$. 
In addition, for any vector $\bx=(x_1,\dots, x_d)\in \R^d$, we set $\bx^\alpha := x_1^{\alpha_1} \cdots x_d^{\alpha_d}$. 
Moreover, for any smooth function $f(\bx)$, $\partial^{\alpha} f$ is defined by $\displaystyle \partial^{\alpha} f(\bx):= \frac{\partial^{|\alpha|}f}{\partial x_1^{\alpha_1} \cdots \partial x_d^{\alpha_d}}(\bx)$. 
\item
For $f\in C^{2}(\Omega)$, the Hessian matrix of $f(\bx)$ is defined by $\displaystyle D^{2}f(\bx) :=\left( \frac{\partial^2 f}{\partial x_i \partial x_j} (\bx) \right)_{1\le i,j \le d}$. 
\item
For a matrix $A=(A_{ij})_{1\le i, j \le d}$, we use the Hilbert--Schmidt norm: $\displaystyle \|A\|_{2}:=\sqrt{\sum_{i=1}^{d}\sum_{j=1}^{d}|A_{ij}|^{2}}$. 
\item
Let $k\in \mathbb{Z}_{\ge0}$. Then, we define $C^{k}(\overline{\Omega})$ as the space of all functions $f : \overline{\Omega} \to \R$ such that $f \in C^{k}(\Omega)$ and $\partial^{\alpha}f$ are uniformly continuous on $\Omega$ for any multi-index $\alpha\in \Z_{\ge 0}^d$ satisfying $|\alpha|\le k$. 
By virtue of the uniform continuity, each $\partial^{\alpha}f$ admits a unique continuous extension to $\overline{\Omega}$. 
 In what follows, we identify $\partial^{\alpha} f$ with its continuous extension to $\overline{\Omega}$. 
 For such functions $\partial^{\alpha}f$, we sometimes take $\|\partial^{\alpha}f\|_{C^{0}(\overline{\Omega})}:=\sup_{\bx \in \overline{\Omega}}|\partial^{\alpha}f(\bx)|$ throughout this paper. 
\end{itemize}

\medskip
\par\noindent
\textbf{\bf{Main Result}} 

\smallskip
Before stating our main result, we shall introduce some assumptions on the functions appearing in the integral \eqref{eq:Laplace_integral}. 
First, let us discuss $h_{n}$, $h$ and $\varepsilon_{n}$. In what follows, assume that $h, \sigma \in C^{3}(\overline{\Omega})$ and the following conditions hold: 

\smallskip
\noindent
\underline{\bf Assumption (A).}
\begin{description}
\item[(i)]
For $p>1$, $\varepsilon_{n} >0$ satisfy $\varepsilon_{n}=O\left(n^{-p}\right)$ as $n\to \infty$. 
\item[(ii)]
$h: \overline{\Omega} \to \R$ has a maximum only at $\bx = \bc \in \Omega$ and satisfies $\det D^{2}h(\bc) \neq 0$. 
\item[(iii)]
$h_{n}: \overline{\Omega} \to \R$ has a maximum only at $\bx = \bc_{n} \in \Omega$ and satisfy $\det D^{2}h_{n}(\bc_{n}) \neq 0$ for all $n\in \mathbb{N}$. 
\item[(iv)]
There exists $N_{0} \in \mathbb{N}$ such that 
\begin{equation}\label{D2-bound}
\inf_{n\ge N_{0}, \, x\in \overline{\Omega}}\left|\det D^{2}h_{n}(\bx)\right|>0. 
\end{equation}
\item[(v)]
There exists $N_{1} \in \mathbb{N}$ such that 
\begin{equation}\label{unif-bound}
\sup_{n\ge N_{1}}\left\|\partial^{\alpha}h_{n}\right\|_{C^{0}(\overline{\Omega})}<\infty 
\end{equation}
holds for any multi-index $\alpha \in \Z_{\ge0}^d$ satisfying $|\alpha|=2, 3$. 
\item[(vi)]
Let $\lambda_{1,n}(\bc_{n}), \cdots, \lambda_{d,n}(\bc_{n})$ be the eigenvalues of $D^{2}h_{n}(\bc_{n})$. Then, there exists $N_{2} \in \mathbb{N}$ such that 
\begin{equation}\label{assump-bound}
C_{\dag}:=\inf_{n\ge N_{2}}\sum_{i=1}^{d}\left|\lambda_{i,n}(c_{n})\right|^{2}>0. 
\end{equation}
\end{description}

Next, we shall explain the assumptions on $g$ as follows: 

\smallskip
\noindent
\underline{\bf Assumption (B).}

\smallskip
\noindent
Let $k \in 2\Z_{\ge0}$ and $g\in C^{k+1}(\overline{\Omega})$. In addition, if $k\ge1$, assume that $g$ satisfies $\partial^{\alpha}g(\bc)=0$ for any multi-index $\alpha \in \Z_{\ge0}^d$ with $|\alpha|\le k-1$ and $\partial^{\alpha}g(\bc)\neq 0$ for some multi-index $\alpha  \in (2\Z_{\ge0})^d$ with $|\alpha|=k$. On the other hand, we also assume $g(\bc) \neq 0$ in the case of $k=0$. 

\medskip
Now, we would like to state our main result on the asymptotic behavior of $I_{n}$ as $n\to \infty$:  
\begin{thm}\label{thm:main}
{\it 
Let $\Omega$ be a bounded domain in $\R^d$ with $d \in \N$. 
Suppose that $h_{n} : \overline{\Omega} \to \R$ are given by the form $h_{n}(\bx)=h(\bx)+\varepsilon_{n}\sigma(\bx)$ for all $n\in \mathbb{N}$, where $\varepsilon_{n} >0$ and $h, \sigma : \overline{\Omega} \to \R$ are satisfy $h, \sigma \in C^{3}(\overline{\Omega})$. 
In addition, assume that $h_{n}$, $h$ and $\varepsilon_{n}$ satisfy the assumption {\bf (A)}. 
Moreover, suppose that the assumption {\bf (B)} is satisfied for $g : \overline{\Omega} \to \R$. Then, $I_{n}$ defined in \eqref{eq:Laplace_integral} satisfies the following asymptotic formula:
\begin{align}\label{asymp}
I_{n}=e^{nh(\bc)} \left( n^{-\frac{d}{2}-\frac{k}{2}}\sqrt{\frac{(2\pi)^d}{|\det D^{2}h(\bc)|}}\hspace{-5mm} \sum_{\substack{|\beta|=k\\ \beta=(\beta_1,\dots, \beta_d) \in (2\Z_{\ge0})^d}} \hspace{-5mm} \left(\partial^\beta g(\bc) \right)\prod_{i=1}^d \frac{|\lambda_i(\bc)|^{-\frac{\beta_i}{2}}}{\beta_i !!}+ O\left(n^{-q}\right)\right)
\end{align}
as $n\to \infty$, where $\lambda_{1}(\bc), \cdots, \lambda_{d}(\bc)$ are the eigenvalues of $D^{2}h(\bc)$ and $\beta=(\beta_1,\dots, \beta_d) \in (2\Z_{\ge0})^d$ is a multi-index, while the exponent $q=q(p, d, k)>0$ is defined by
\begin{equation}
q:=\begin{cases}
\displaystyle \frac{d}{2}+\frac{k}{2}+p-1, &\text{if} \quad \displaystyle 1<p<\frac{3}{2}, \\[1em]
\displaystyle \frac{d}{2}+\frac{k}{2}+\frac{1}{2}, &\text{if} \quad \displaystyle p\ge \frac{3}{2}. 
\end{cases}
\end{equation} 
}
\end{thm}

\begin{rem}
{\it 
We note that our result is also valid for $\sigma(\bx)\equiv 0$. 
In this case, the assumption {\bf (B)} and only the condition {\bf (ii)} in the assumption {\bf (A)} are required to prove Theorem {\rm \ref{thm:main}}, and the other conditions are not needed. 
Under this situation, the exponent $p>1$ in the condition {\bf (i)} is interpreted as $p=\infty$. 
Therefore, the error bound in \eqref{asymp} is $O(n^{-\frac{d}{2}-\frac{k}{2}-\frac{1}{2}})$, which is the same result as for $p\ge3/2$.
The fact that the same results are obtained when $p\ge3/2$ as when $\sigma(\bx)\equiv 0$ means that $\varepsilon_{n}\sigma(\bx)$ can be considered as perturbations if $p\ge3/2$. 
On the other hand, in the case of $1<p<3/2$, although $\varepsilon_{n} \to 0$, the term $\varepsilon_{n}\sigma(\bx)$ cannot be regarded as perturbations. 
Moreover, it should be noted that the decay order of $\varepsilon_{n}$ has a non-negligible effect on the asymptotic behavior. 
Furthermore, in the case of $0<p<1$, it does not even have meaning as an asymptotic formula, since the error bound $O(n^{-\frac{d}{2}-\frac{k}{2}-(p-1)})$ decays slower than the leading term.
}
\end{rem}

\begin{rem}
{\it 
Let $\sigma(\bx)\equiv 0$. If $k=0$, our asymptotic profile given in \eqref{asymp} is consistent with the classical result \eqref{thm-Laplace-1}. 
Moreover, for the one-dimensional case $d=1$, our formula is a directly follows from the results given in {\rm \cite{FK24, N20, O94, W19, WU24}}. 
On the other hand, we emphasize that our formula is an extension to the multidimensional version of the one-dimensional result by them. 
In addition, another multidimensional formula with an error estimate is also given by Kirwin {\rm \cite{K10}}. 
He gave another asymptotic profile with an error bound $O(n^{-\frac{d}{2}-\frac{k}{2}+1})$ under some assumptions different from ours. On the other hand, compared to his result, our result \eqref{asymp} gives an improved error bound $O(n^{-\frac{d}{2}-\frac{k}{2}-\frac{1}{2}})$. 
Moreover, a result related to {\rm \cite{K10}} is also obtained in {\rm \cite{KP}}. However, the expression of the remainder term is different from {\rm \cite{K10}} and ours.
}
\end{rem}

\begin{rem}
{\it 
As we mentioned in the above, \L{}api\'{n}ski {\rm \cite{L19}} studied the Laplace-type integral where the phase function depending on $n\in \mathbb{N}$ and obtained an asymptotic formula with an error estimate, under appropriate assumptions (see, also {\rm \cite{K20}}). 
Our result can be regarded as a higher-order asymptotic expansion version of his result. 
However, it should be noted that the assumptions and methods used in the proof are different. 
In addition, we would like to emphasize that our result successfully captures the effect of the decay rate of the perturbations $\varepsilon_{n}=O\left(n^{-p}\right)$ on the error bound $O\left(n^{-q}\right)$ in the asymptotic formula.
}
\end{rem}

\section{Matrix Analysis under Perturbations}

\indent

In this section, we prepare results on matrix analysis under perturbations to prove the main result. More precisely, for the Hessian matrices $D^{2}h_{n}(\bc_{n})$ and $D^{2}h(\bc)$, we investigate the relationship between their determinants $\det D^{2}h_{n}(\bc_{n})$ and $\det D^{2}h(\bc)$, as well as the relationship between their eigenvalues $\lambda_{i,n}(\bc_{n})$ and $\lambda_{i}(\bc)$. 
In order to do that, let us introduce the following lemma on an approximation for $\bc_{n}$ of the maximum point of $h_{n}$: 
\begin{lem}\label{lem:cn-c}
{\it
Let $\Omega$ be a bounded domain in $\R^d$ with $d \in \N$. 
Suppose that $h_{n} : \overline{\Omega} \to \R$ are given by the form $h_{n}(\bx)=h(\bx)+\varepsilon_{n}\sigma(\bx)$ for all $n\in \mathbb{N}$, where $\varepsilon_{n} >0$ and $h, \sigma : \overline{\Omega} \to \R$ are satisfy $h, \sigma \in C^{2}(\overline{\Omega})$. 
Assume that $h_{n}$, $h$ and $\varepsilon_{n}$ satisfy the conditions {\bf (i)}, {\bf (ii)}, {\bf (iii)} and {\bf (iv)} in the assumption {\bf (A)}, and \eqref{unif-bound} holds for only $|\alpha|=2$. Then, we have 
\begin{equation}\label{cn-c}
\left|\bc_{n}-\bc\right|=O\left(n^{-p}\right) \ \ \text{as} \ \ n\to \infty. 
\end{equation}
}
\end{lem}
\begin{proof}
First, applying Taylor's theorem to $\nabla h_{n}$ around $\bc$, there exists $\theta=\theta(\bx, \bc, n) \in (0, 1)$ such that 
\[
\nabla h_{n}(\bx)=\nabla h_{n}(\bc)+D^{2}h_{n}(\bc+\theta(\bx-\bc))(\bx-\bc). 
\]
From the condition {\bf (iii)}, we get $\nabla h_{n}(\bc_{n})=\bm{0}$. 
Therefore, substituting $\bx=\bc_{n}$ into the above, we obtain 
\begin{equation*}
D^{2}h_{n}(\bc+\theta(\bc_{n}-\bc))(\bc_{n}-\bc)=-\nabla h_{n}(\bc) \quad \text{for some} \quad \theta=\theta(\bc_{n}, \bc, n) \in (0, 1).
\end{equation*}
On the other hand, since $\nabla h_{n}(\bx)=\nabla h(\bx)+\varepsilon_{n}\nabla \sigma(\bx)$ and $\nabla h(\bc)=\bm{0}$ from the condition {\bf (ii)}, we have
\begin{equation*}
\nabla h_{n}(\bc)=\nabla h(\bc)+\varepsilon_{n}\nabla \sigma(\bc)=\varepsilon_{n}\nabla \sigma(\bc). 
\end{equation*}
Thus, combining the above two results and taking the inverse matrix of $D^{2}h_{n}(\bc+\theta(\bc_{n}-\bc))$, we get 
\begin{align}
\bc_{n}-\bc&=-\varepsilon_{n}\left\{D^{2}h_{n}(\bc+\theta(\bc_{n}-\bc))\right\}^{-1}\nabla \sigma(\bc) \nonumber \\
&=-\varepsilon_{n}\frac{\text{adj}\left(D^{2}h_{n}(\bc+\theta(\bc_{n}-\bc))\right)}{\det D^{2}h_{n}(\bc+\theta(\bc_{n}-\bc))}\nabla \sigma(\bc)  
=:\varepsilon_{n}\ba_{n}=\varepsilon_{n}(a_{n}^{(1)}, \cdots, a_{n}^{(d)}), \label{DEF-R}
\end{align}
where $\text{adj}\left(D^{2}h_{n}(\bx)\right)$ denotes the adjugate matrix of $D^{2}h_{n}(\bx)$. 
(The above notation $\ba_{n}$ will be used in the proof of Theorem \ref{thm:main}). Therefore, the desired result \eqref{cn-c} can be derived by the generalized Cauchy--Schwarz inequality. Actually, if $n\in \mathbb{N}$ is sufficiently large, it follows from \eqref{D2-bound} and \eqref{unif-bound} that 
\begin{equation}\label{R-est}
\left|\ba_{n}\right| \le \frac{\displaystyle \sup_{n\ge N_{1}, \, x\in \Omega}\left\|\text{adj}\left(D^{2}h_{n}(\bx)\right)\right\|_{2}}{\displaystyle \inf_{n\ge N_{0}, \, x\in \Omega}\left|\det D^{2}h_{n}(\bx)\right|} \left\|\nabla \sigma\right\|_{C^{0}(\overline{\Omega})}
=O\left(1\right),  
\end{equation}
where we implicitly used the fact that the elements of the adjugate matrix of $D^{2}h_{n}(\bx)$ consist of products or combinations of second derivatives of $h_{n}(\bx)$. 
By summarizing \eqref{DEF-R}, \eqref{R-est} and the condition {\bf (i)}, the proof of \eqref{cn-c} is complete. 
\end{proof}

Next, we shall derive a result on the approximation of the determinant. 
The following proposition shows that the asymptotic profile of $\det D^{2}h_{n}(\bc_{n})$ can be given by $\det D^{2}h(\bc)$ as $n\to \infty$: 
\begin{prop}\label{lem:det}
{\it 
Let $\Omega$ be a bounded domain in $\R^d$ with $d \in \N$. 
Suppose that $h_{n} : \overline{\Omega} \to \R$ are given by the form $h_{n}(\bx)=h(\bx)+\varepsilon_{n}\sigma(\bx)$ for all $n\in \mathbb{N}$, where $\varepsilon_{n} >0$ and $h, \sigma : \overline{\Omega} \to \R$ are satisfy $h, \sigma \in C^{2}(\overline{\Omega})$. 
Assume that $h_{n}$, $h$ and $\varepsilon_{n}$ satisfy the conditions {\bf (i)}, {\bf (ii)}, {\bf (iii)} and {\bf (iv)} in the assumption {\bf (A)}, and \eqref{unif-bound} holds for only $|\alpha|=2$. Then, we have 
\begin{equation}\label{det-expan-prop}
\det D^{2}h_{n}(\bc_{n})=\det D^{2}h(\bc)+O\left(n^{-p}\right),  \ \ \text{as} \ \ n\to \infty. 
\end{equation}
}
\end{prop}
\begin{proof}
From the definition of $h_{n}$, \eqref{DEF-R} and Taylor's theorem, there exists $\theta=\theta(\bc, n)\in (0, 1)$ such that 
\begin{align}
\left(D^{2}h_{n}(\bc_{n})\right)_{ij}
&=\frac{\partial^{2}h_{n}}{\partial x_{i} \partial x_{j}}(\bc_{n})
=\frac{\partial^{2}h_{n}}{\partial x_{i} \partial x_{j}}(\bc+\varepsilon_{n} \ba_{n})
=\frac{\partial^{2}h}{\partial x_{i} \partial x_{j}}(\bc+\varepsilon_{n} \ba_{n})+\varepsilon_{n}\frac{\partial^{2}\sigma}{\partial x_{i} \partial x_{j}}(\bc+\varepsilon_{n} \ba_{n}) \nonumber \\
&=\frac{\partial^{2}h}{\partial x_{i} \partial x_{j}}(\bc)+\varepsilon_{n}\ba_{n}\cdot \frac{\partial^{2}}{\partial x_{i} \partial x_{j}}\nabla h(\bc+\theta\varepsilon_{n} \ba_{n})+\varepsilon_{n}\frac{\partial^{2}\sigma}{\partial x_{i} \partial x_{j}}(\bc+\varepsilon_{n} \ba_{n}) \nonumber \\
&=:\frac{\partial^{2}h}{\partial x_{i} \partial x_{j}}(\bc)+\varepsilon_{n}f_{ij}(\bc, n). \label{det-expan-junbi}
\end{align}
For $f_{ij}(\bc, n)$, we note that the following estimate can be easily obtained from \eqref{R-est}: 
\begin{equation}\label{fij-est}
\left|f_{ij}(\bc, n)\right| \le |\ba_{n}|\left\|\frac{\partial^{2}}{\partial x_{i} \partial x_{j}}\nabla h\right\|_{C^{0}(\overline{\Omega})}+\left\|\frac{\partial^{2}\sigma}{\partial x_{i} \partial x_{j}}\right\|_{C^{0}(\overline{\Omega})}=O(1). 
\end{equation}
Moreover, from \eqref{det-expan-junbi} and the definition of determinant, we have
\[
\det D^{2}h_{n}(\bc_{n})
=\sum_{\rho\in S_{d}} \mathrm{sgn}(\rho)\prod_{i=1}^{d}\left(\frac{\partial^{2}h}{\partial x_{i} \partial x_{\rho(i)}}(\bc)+\varepsilon_{n}f_{i\rho(i)}(\bc, n)\right), 
\]
where we denote the set of all permutations of $d$ elements by $S_{d}$. 
Furthermore, by carefully calculating the products of the above equation, we obtain
\begin{align*}
&\prod_{i=1}^{d}\left(\frac{\partial^{2}h}{\partial x_{i} \partial x_{\rho(i)}}(\bc)+\varepsilon_{n}f_{i\rho(i)}(\bc, n)\right) \\
&=\prod_{i=1}^{d}\frac{\partial^{2}h}{\partial x_{i} \partial x_{\rho(i)}}(\bc)
+\sum_{l=1}^{d}\sum_{1\le i_{1}<\cdots <i_{l}\le d}\varepsilon_{n}^{l}f_{i_{1}\rho(i_{1})}(\bc, n) \cdots f_{i_{l}\rho(i_{l})}(\bc, n)\prod_{j\neq i_{1}, \cdots, i_{l}}\frac{\partial^{2}h}{\partial x_{j} \partial x_{\rho(j)}}(\bc) \\
&=\prod_{i=1}^{d}\frac{\partial^{2}h}{\partial x_{i} \partial x_{\rho(i)}}(\bc)
+\varepsilon_{n}\sum_{l=1}^{d}\sum_{1\le i_{1}<\cdots <i_{l}\le d}\varepsilon_{n}^{l-1}f_{i_{1}\rho(i_{1})}(\bc, n) \cdots f_{i_{l}\rho(i_{l})}(\bc, n) \prod_{j\neq i_{1}, \cdots, i_{l}}\frac{\partial^{2}h}{\partial x_{j} \partial x_{\rho(j)}}(\bc) \\
&=:\prod_{i=1}^{d}\frac{\partial^{2}h}{\partial x_{i} \partial x_{\rho(i)}}(\bc)+\varepsilon_{n}F_{n, \rho}. 
\end{align*}
Thus, we arrive at the following expansion: 
\begin{align}
\det D^{2}h_{n}(\bc_{n})
&=\sum_{\rho\in S_{d}} \mathrm{sgn}(\rho)\left(\prod_{i=1}^{d}\frac{\partial^{2}h}{\partial x_{i} \partial x_{\rho(i)}}(\bc)+\varepsilon_{n}F_{n, \rho}\right) \nonumber \\
&=\det D^{2}h(\bc)+\varepsilon_{n}\sum_{\rho\in S_{d}} \mathrm{sgn}(\rho)F_{n, \rho} 
=:\det D^{2}h(\bc)+C_{n}.  \label{Cn-est}
\end{align}
(This notation $C_{n}$ will be used in the proof of Theorem \ref{thm:main}). 
From \eqref{fij-est} and the condition {\bf (i)}, $C_{n}\in \R$ can be evaluated as $C_{n}=O\left(n^{-p}\right)$. 
Therefore, we can conclude that the desired result \eqref{det-expan-prop} is true. 
\end{proof}

Finally, we shall introduce the following result concerning the approximation of the eigenvalues. 
The following proposition means that the eigenvalues $\lambda_{i,n}(\bc_{n})$ of $D^{2}h_{n}(\bc_{n})$ can be asymptotically expanded by the eigenvalues $\lambda_{i}(\bc)$ of $D^{2}h(\bc)$ as $n\to \infty$: 
\begin{prop}\label{lem:eigen}
{\it 
Let $\Omega$ be a bounded domain in $\R^d$ with $d \in \N$. 
Suppose that $h_{n} : \overline{\Omega} \to \R$ are given by the form $h_{n}(\bx)=h(\bx)+\varepsilon_{n}\sigma(\bx)$ for all $n\in \mathbb{N}$, where $\varepsilon_{n} >0$ and $h, \sigma : \overline{\Omega} \to \R$ are satisfy $h, \sigma \in C^{3}(\overline{\Omega})$. 
Assume that $h_{n}$, $h$ and $\varepsilon_{n}$ satisfy the conditions {\bf (i)}, {\bf (ii)}, {\bf (iii)} and {\bf (iv)} in the assumption {\bf (A)}, and \eqref{unif-bound} holds for only $|\alpha|=2$. Then, we have 
\begin{equation}\label{eigen}
\lambda_{i,n}(\bc_{n})=\lambda_{i}(\bc)+O\left(n^{-p}\right) \ \ \text{as} \ \ n\to \infty, 
\end{equation}
where $\lambda_{i,n}(\bc_{n})$ and $\lambda_{i}(\bc)$ are the $i$-th eigenvalues of $D^{2}h_{n}(\bc_{n})$ and $D^{2}h(\bc)$, respectively. 
}
\end{prop}
\begin{proof}
Let us split the difference $\Lambda_{i,n}:=\lambda_{i, n}(\bc_{n})-\lambda_{i}(\bc)$ into the following two parts:
\begin{equation}\label{DEF-Lambda}
\Lambda_{i,n} =\left(\lambda_{i, n}(\bc_{n})-\lambda_{i}(\bc_{n}) \right)+ \left(\lambda_{i}(\bc_{n})-\lambda_{i}(\bc) \right), 
\end{equation}
where $\lambda_{i}(\bc_{n})$ is the $i$-th eigenvalue of $D^{2}h(\bc_{n})$. 
(The above notation $\Lambda_{i,n}$ will be used in the proof of Theorem \ref{thm:main}). First, from the definition of $h_{n}$, we have 
\begin{align*}
D^{2}h_{n}(\bc_{n})
&=\left(\frac{\partial^{2}h_{n}}{\partial x_{i} \partial x_{j}}(\bc_{n})\right)_{1\le i, j \le d}
=\left(\frac{\partial^{2}h}{\partial x_{i} \partial x_{j}}(\bc_{n})\right)_{1\le i, j \le d}+\left(\varepsilon_{n}\frac{\partial^{2}\sigma}{\partial x_{i} \partial x_{j}}(\bc_{n})\right)_{1\le i, j \le d}. 
\end{align*}
Then, since all of the above terms are symmetric matrices, it follows from Wely's perturbation theorem (cf. Corollary III.2.6 in \cite{B97}, see also \cite{GV00}) and the condition {\bf (i)} that 
\begin{align}
\left|\lambda_{i, n}(\bc_{n})-\lambda_{i}(\bc_{n})\right|
\le \varepsilon_{n}\left\|D^{2}\sigma(\bc_{n})\right\|_{2}
\le \varepsilon_{n}\sqrt{\sum_{i=1}^{d}\sum_{j=1}^{d} \left\|\frac{\partial^{2}\sigma}{\partial x_{i} \partial x_{j}}\right\|_{C^{0}(\overline{\Omega})}^{2}}
=O\left(n^{-p}\right). \label{eigen-lem-1}
\end{align}
Moreover, using Wely's perturbation theorem again, from Taylor's theorem and Lemma \ref{lem:cn-c}, there exists $\theta=\theta(\bc_{n}, \bc)\in (0, 1)$ such that  
\begin{align}
\left|\lambda_{i}(\bc_{n})-\lambda_{i}(\bc)\right|
&\le \left\|D^{2}h(\bc_{n})-D^{2}h(\bc)\right\|_{2} 
=\sqrt{   \sum_{i=1}^{d}\sum_{j=1}^{d}  \left| \frac{\partial^{2}h}{\partial x_{i} \partial x_{j}}(\bc_{n}) -\frac{\partial^{2}h}{\partial x_{i} \partial x_{j}}(\bc)  \right|^{2}} \nonumber \\
&=\sqrt{   \sum_{i=1}^{d}\sum_{j=1}^{d}  \left| \frac{\partial^{2}}{\partial x_{i} \partial x_{j}}\nabla h(\bc+\theta(\bc_{n}-\bc))(\bc_{n}-\bc) \right|^{2}} \nonumber \\
&\le \left|\bc_{n}-\bc\right|\sqrt{\sum_{i=1}^{d}\sum_{j=1}^{d}   \left\| \frac{\partial^{2}}{\partial x_{i} \partial x_{j}}\nabla h\right\|_{C^{0}(\overline{\Omega})}^{2}   } 
=O\left(n^{-p}\right). \label{eigen-lem-2}
\end{align}
Finally, combining \eqref{DEF-Lambda}, \eqref{eigen-lem-1} and \eqref{eigen-lem-2}, we obtain $\Lambda_{i, n}=O\left(n^{-p}\right)$. This completes the proof.  
\end{proof}

\section{Proof of the Main Result}

\indent

Before proving Theorem \ref{thm:main}, we would like to introduce a simple lemma which will be used in that proof. 
That is, the following basic result on the integration of the Gaussian function on $\R^d$ (we omit its proof, since it can be given by a standard calculation): 
\begin{lem}\label{lem:gaussian}
{\it
Let $\by=(y_1,\dots, y_d)\in \R^d$ and $A=(A_{ij})_{1\le i, j \le d}$ be a negative definite symmetric matrix. 
Then, for any multi-index $\beta=(\beta_1,\dots, \beta_d)\in (2\Z_{\ge0})^d$, the following formula holds: 
$$
\int_{\R^d} \exp\left(\frac{1}{2}\sum_{i=1}^d\sum_{j=1}^d A_{ij} y_iy_j \right) \by^\beta d\by 
=\prod_{i=1}^d \left( \frac{|\lambda_i|}{2}\right)^{-\frac{\beta_i+1}{2}} \Gamma\left( \frac{\beta_i +1}{2}\right), 
$$
where $\lambda_1,\dots, \lambda_d<0$ are the eigenvalues of $A$, while $\Gamma$ is the Gamma function. 
}
\end{lem}

Finally, we give the proof of our main result Theorem \ref{thm:main}:

\medskip
\noindent
{\bf \bf Proof of Theorem \ref{thm:main}.}

\smallskip
First of all, we note that $h_{n}\in C^{3}(\overline{\Omega})$ due to $h,\sigma\in C^{3}(\overline{\Omega})$. In addition, from the condition {\bf (iii)}, we get $\nabla h_{n}(\bc_{n})=0$. 
Therefore, applying Taylor's theorem to $h_{n}$, there exists $\theta_0=\theta_{0}(\bx, \bc_{n}, n)\in (0,1)$ such that
\begin{align}\label{eq:h}
h_{n}(\bx)=h_{n}(\bc_{n})  + \sum_{|\alpha|=2} \frac{1}{\alpha!} \partial^\alpha h_{n}(\bc_{n})(\bx-\bc_{n})^\alpha + \sum_{|\alpha|=3} \frac{1}{\alpha !} \partial^\alpha h_{n}(\bc_{n}+\theta_0(\bx-\bc_{n}))(\bx-\bc_{n})^\alpha.
\end{align}

In what follows, let us derive the asymptotic profile of $I_{n}$. 
Before doing that, we shall find the leading term of $I_{n}e^{-nh_{n}(\bc_{n})}$. 
Now, we take sufficiently small $\delta>0$. We start with dividing $I_{n}e^{-nh_{n}(\bc_{n})}$ into the following two integrals $J_n^{(1)}$ and $J_n^{(2)}$:
\begin{align}\label{J1+J2}
I_{n}e^{-nh_{n}(\bc_{n})}
= \left(\int_{\Omega \setminus B_\delta(\bc)} + \int_{B_\delta(\bc)} \right) e^{n(h_{n}(\bx)-h_{n}(\bc_{n}))}g(\bx)d\bx 
=: J_n^{(1)} + J_n^{(2)}, 
\end{align}
where $B_{\delta}(\bc)$ is defined by $B_{\delta}(\bc):= \{\bx \in \R^d: |\bx - \bc | < \delta\}$. 

First, let us deal with $J_n^{(1)}$ in \eqref{J1+J2}. In order to do that, we shall prepare some an auxiliary estimate. 
It follows from the definition of $h_{n}$ that 
\begin{align}\label{h-junbi-1}
h_{n}(\bx)-h_{n}(\bc_{n})=h(\bx)+\varepsilon_{n}\sigma(\bx)-h(\bc_{n})-\varepsilon_{n}\sigma(\bc_{n})
=h(\bx)-h(\bc_{n})+\varepsilon_{n}\left\{ \sigma(\bx) -\sigma(\bc_{n})\right\}. 
\end{align}
Applying Taylor's theorem to $h$ and $\sigma$, there exist $\theta_1=\theta_{1}(\bc_{n}, \bc, n)\in (0,1)$ and $\theta_2=\theta_{2}(\bc_{n}, \bc, n)\in (0,1)$, respectively, such that
\begin{align}
\begin{split}\label{h-junbi-2}
h(\bc_{n})=h(\bc)+\nabla h(\bc+\theta_{1}(\bc_{n}-\bc))(\bc_{n}-\bc), \quad  
\sigma(\bc_{n})=\sigma(\bc)+\nabla \sigma (\bc+\theta_{2}(\bc_{n}-\bc))(\bc_{n}-\bc). 
\end{split}
\end{align}
Then, combining \eqref{h-junbi-1} and \eqref{h-junbi-2}, we have
\begin{align}
n(h_{n}(\bx)-h_{n}(\bc_{n}))
&=n\bigl[h(\bx)-h(\bc)-\nabla h(\bc+\theta_{1}(\bc_{n}-\bc))(\bc_{n}-\bc) \nonumber \\
&\qquad +\varepsilon_{n}\left\{ \sigma(\bx)-\sigma(\bc)-\nabla \sigma (\bc+\theta_{2}(\bc_{n}-\bc))(\bc_{n}-\bc) \right\}\bigl] \nonumber \\
&=-n(h(\bc)-h(\bx))+n\varepsilon_{n}\left(\sigma(\bx)-\sigma(\bc)\right)+R_{n},  \label{J1-junbi}
\end{align}
where a new remainder term $R_{n}$ is defined by 
\[
R_{n}:=-n\nabla h(\bc+\theta_{1}(\bc_{n}-\bc))(\bc_{n}-\bc) 
-n\varepsilon_{n} \nabla \sigma (\bc+\theta_{2}(\bc_{n}-\bc))(\bc_{n}-\bc). 
\]
Note that the following error bound can be obtained from Lemma \ref{lem:cn-c} and the condition {\bf (i)}: 
\begin{equation}\label{new-R-est}
\left|R_{n}\right|\le n|\bc_{n}-\bc|\left\|\nabla h\right\|_{C^{0}(\overline{\Omega})}+n\varepsilon_{n}|\bc_{n}-\bc|\left\|\nabla \sigma \right\|_{C^{0}(\overline{\Omega})}
=O\left(n^{-(p-1)}\right). 
\end{equation}
Eventually, we would like to evaluate $J_n^{(1)}$ in \eqref{J1+J2}. If we put 
\[
A:= \min_{\bx \in \overline{\Omega} \setminus B_\delta(\bc)} \{h(\bc)-h(\bx)\} >0, 
\]
then, it can be easily obtained by virtue of \eqref{J1-junbi}, \eqref{new-R-est} and the fact $n\varepsilon_{n}=O\left(n^{-(p-1)}\right)$ as follows: 
\begin{align}
\left|J_n^{(1)}\right| 
&\le \int_{\Omega \setminus B_\delta(\bc)} e^{-n(h(\bc)-h(\bx))+n\varepsilon_{n}\left(\sigma(\bx)-\sigma(\bc)\right)+R_{n}}\left|g(\bx)\right| d\bx  \nonumber \\
&\le e^{-nA}\exp\left(2n\varepsilon_{n}\left\|\sigma\right\|_{C^{0}(\overline{\Omega})}+|R_{n}|\right)\int_{\Omega \setminus B_\delta(\bc)} |g(\bx)|d\bx 
= O\left(e^{-An}\right). \label{est-J2}
\end{align}

Next, let us treat $J_n^{(2)}$ in \eqref{J1+J2} and extract the leading term of $I_{n}e^{-nh_{n}(\bc_{n})}$ from this part. 
In what follows, let us take $n\in \mathbb{N}$ sufficiently large. Now, recalling the assumption {\bf (B)} and applying Taylor's theorem to $g$, then there exists $\theta_3=\theta_{3}(\bx, \bc, k)\in (0,1)$ such that
\begin{align}\label{eq:g}
g(\bx)= \sum_{|\beta|=k} \frac{1}{\beta!}\partial^\beta g(\bc) (\bx-\bc)^\beta +  \sum_{|\beta|=k+1} \frac{1}{\beta!}\partial^\beta g(\bc+\theta_3(\bx-\bc)) (\bx-\bc)^\beta. 
\end{align}
Then, it follows from \eqref{DEF-R}, \eqref{eq:h}, \eqref{eq:g} and the change of variable  $\bx-\bc= \by/\sqrt{n}$ that
\begin{align}
J_n^{(2)}
&= \int_{B_\delta(\bc)}\left\{ \sum_{|\beta|=k} \frac{1}{\beta!}\partial^\beta g(\bc) (\bx-\bc)^\beta +  \sum_{|\beta|=k+1} \frac{1}{\beta!}\partial^\beta g(\bc+\theta_3(\bx-\bc)) (\bx-\bc)^\beta \right\} \nonumber \\
&\ \ \ \times 
 \exp\left\{n\left( \sum_{|\alpha|=2} \frac{1}{\alpha !} \partial^\alpha h_{n}(\bc_{n})(\bx-\bc_{n})^\alpha +\sum_{|\alpha|=3} \frac{1}{\alpha !} \partial^\alpha h_{n}(\bc_{n}+\theta_0(\bx-\bc_{n}))(\bx-\bc_{n})^\alpha \right) \right\}d\bx \nonumber\\
&=n^{-\frac{d}{2}}\int_{B_{\delta\sqrt{n}}(0)} 
\left\{ \sum_{|\beta|=k} \frac{1}{\beta!}\partial^\beta g(\bc) \left(\frac{\by}{\sqrt{n}}\right)^\beta +  \sum_{|\beta|=k+1} \frac{1}{\beta!}\partial^\beta g\left(\bc+\theta_3\frac{\by}{\sqrt{n}}\right) \left(\frac{\by}{\sqrt{n}}\right)^\beta \right\} \nonumber\\
&\ \ \ \times  \exp \left\{n\sum_{|\alpha|=3} \frac{1}{\alpha !} \partial^\alpha h_{n}\left(\bc_{n}+\theta_0\frac{\by}{\sqrt{n}} -\theta_{0}\varepsilon_{n}\ba_{n}\right)\left(\frac{\by}{\sqrt{n}}-\varepsilon_{n}\ba_{n}\right)^\alpha \right\}  \nonumber \\
&\ \ \ \times \exp\left(\sum_{|\alpha|=2} \frac{1}{\alpha !} \partial^\alpha h_{n}(\bc_{n})\left(\by-\varepsilon_{n}\sqrt{n}\ba_{n}\right)^\alpha\right)d\by. \label{Jn2-junbi}
\end{align}
Here, we would like to analyze $\left(\by-\varepsilon_{n}\sqrt{n}\ba_{n}\right)^\alpha$ for $|\alpha|=2$, in the above. 
It can be divided into the two cases {\bf (1)} and {\bf (2)} below. 
Actually, the following results can be easily obtained: 

\smallskip
\noindent
{\bf (1)} If the $i$-th and $j$-th components are $1$ and the others are 0, i.e., $\alpha=(0, \cdots, 1, \cdots, 1, \cdots, 0)$, 
\begin{align*}
\left(\by-\varepsilon_{n}\sqrt{n}\ba_{n}\right)^\alpha
&=\left(y_{i}-\varepsilon_{n}\sqrt{n}a_{n}^{(i)}\right)\left(y_{j}-\varepsilon_{n}\sqrt{n}a_{n}^{(j)}\right) \\
&=y_{i}y_{j}-\varepsilon_{n}\sqrt{n}a_{n}^{(i)}y_{j}-\varepsilon_{n}\sqrt{n}a_{n}^{(j)}y_{i}+\varepsilon_{n}^{2}na_{n}^{(i)}a_{n}^{(j)}
=\by^{\alpha}+O\left(n^{-(p-1)}\right). 
\end{align*}
\noindent
{\bf (2)} If only the $i$-th component is $2$ and the others are 0, i.e., $\alpha=(0, \cdots, 2, \cdots, 0)$,
\begin{align*}
\left(\by-\varepsilon_{n}\sqrt{n}\ba_{n}\right)^\alpha
=\left(y_{i}-\varepsilon_{n}\sqrt{n}a_{n}^{(i)}\right)^{2}
=y_{i}^{2}-2\varepsilon_{n}\sqrt{n}a_{n}^{(i)}y_{i}+\varepsilon_{n}^{2}n\left\{a_{n}^{(i)}\right\}^{2}
=\by^{\alpha}+O\left(n^{-(p-1)}\right), 
\end{align*}
where we used the facts \eqref{R-est}, $|y_{i}|\le \delta \sqrt{n}$ for all $i=1, \cdots, d$, $n\varepsilon_{n}=O\left(n^{-(p-1)}\right)$ and $n\varepsilon_{n}^{2}=O\left(n^{-(2p-1)}\right)$. Therefore, in both the above cases, there exists $L_{n}(\by, \alpha)\in \R$ such that the following result holds: 
\begin{equation}\label{L-est}
\left(\by-\varepsilon_{n}\sqrt{n}\ba_{n}\right)^\alpha=\by^{\alpha}+ L_{n}(\by, \alpha) \quad \text{with} \quad L_{n}(\by, \alpha)=O\left(n^{-(p-1)}\right), \quad \text{for} \quad |\alpha|=2. 
\end{equation}
By using this fact and Taylor's theorem, there exists $\theta_4=\theta_{4}(\by, \bc_{n}, n) \in (0,1)$ such that
\begin{align}
&\exp\left\{\sum_{|\alpha|=2} \frac{1}{\alpha !} \partial^\alpha h_{n}(\bc_{n})\left(\by-\varepsilon_{n}\sqrt{n}\ba_{n}\right)^\alpha\right\} 
=\exp\left\{\sum_{|\alpha|=2} \frac{1}{\alpha !} \partial^\alpha h_{n}(\bc_{n})(\by^{\alpha}+L_{n}(\by, \alpha))\right\}  \nonumber \\
&=\exp\left\{ \frac{1}{2}\sum_{i=1}^{d}\sum_{j=1}^{d}\left(D^{2}h_{n}(\bc_{n})\right)_{ij}y_{i}y_{j} \right\} \exp\left(\sum_{|\alpha|=2} \frac{1}{\alpha !} \partial^\alpha h_{n}(\bc_{n})L_{n}(\by, \alpha)\right) \nonumber \\
&=\exp\left\{ \frac{1}{2}\sum_{i=1}^{d}\sum_{j=1}^{d}\left(D^{2}h_{n}(\bc_{n})\right)_{ij}y_{i}y_{j}\right\} 
\left(1+M_{n}(\by)\right), \label{Jn2-junbi-2}
\end{align}
where a new remainder term $M_{n}(\by)$ is defined by 
\[
M_{n}(\by):=\left( \sum_{|\alpha|=2} \frac{1}{\alpha !} \partial^\alpha h_{n}(\bc_{n}) L_{n}(\by, \alpha)\right) \exp\left( \theta_{4}\sum_{|\alpha|=2} \frac{1}{\alpha !} \partial^\alpha h_{n}(\bc_{n})L_{n}(\by, \alpha)\right). 
\]
We note that the following error bound of $M_{n}(\by)$ can be obtained by \eqref{L-est} and the condition {\bf (v)}: 
\begin{align}
\left\|M_{n}\right\|_{C^{0}(\overline{\Omega})}
&\le \left(\sum_{|\alpha|=2}\frac{1}{\alpha!}\sup_{n\ge N_{1}}\left\|\partial^{\alpha}h_{n}\right\|_{C^{0}(\overline{\Omega})}\left\|L_{n}(\cdot,\alpha)\right\|_{C^{0}(\overline{\Omega})}\right) \nonumber \\
&\ \ \ \times \exp\left(\sum_{|\alpha|=2}\frac{1}{\alpha!}\sup_{n\ge N_{1}}\left\|\partial^{\alpha}h_{n}\right\|_{C^{0}(\overline{\Omega})}\left\|L_{n}(\cdot,\alpha)\right\|_{C^{0}(\overline{\Omega})}\right) 
=O\left(n^{-(p-1)}\right). \label{Mn-est}
\end{align}
Moreover, using Taylor's theorem again, there exists $\theta_5=\theta_{5}(\by, \bc_{n}, n) \in (0,1)$  such that
\begin{align}
&\exp \left\{n\sum_{|\alpha|=3} \frac{1}{\alpha !} \partial^\alpha h_{n}\left(\bc_{n}+\theta_0\frac{\by}{\sqrt{n}} -\theta_{0}\varepsilon_{n}\ba_{n}\right)\left(\frac{\by}{\sqrt{n}}-\varepsilon_{n}\ba_{n}\right)^\alpha \right\}  \nonumber \\
&=1+ n^{-\frac{1}{2}}\sum_{|\alpha|=3} \frac{1}{\alpha!} \partial^\alpha h_{n} \left(\bc_{n}+\theta_0\frac{\by}{\sqrt{n}} -\theta_{0}\varepsilon_{n}\ba_{n}\right) \left(\by-\varepsilon_{n}\sqrt{n}\ba_{n}\right)^\alpha \nonumber \\
&\ \ \ \times \exp\left\{\frac{\theta_5}{\sqrt{n}} \sum_{|\alpha|=3} \frac{1}{\alpha!} \partial^\alpha h_{n} \left(\bc_{n}+\theta_0\frac{\by}{\sqrt{n}} -\theta_{0}\varepsilon_{n}\ba_{n}\right) \left(\by-\varepsilon_{n}\sqrt{n}\ba_{n}\right)^\alpha \right\}. \label{Jn2-junbi-3}
\end{align}
Therefore, by combining \eqref{Jn2-junbi}, \eqref{Jn2-junbi-2} and \eqref{Jn2-junbi-3}, and rearranging the resulting expression, $J_n^{(2)}$ can be decomposed into the following three integrals using $J_n^{(2,1)}$, $J_n^{(2,2)}$ and $J_n^{(2,3)}$: 
\begin{align}
J_n^{(2)} &=n^{-\frac{d}{2}}\int_{B_{\delta\sqrt{n}}(0)}
\exp\left\{ \frac{1}{2}\sum_{i=1}^{d}\sum_{j=1}^{d}\left(D^{2}h_{n}(\bc_{n})\right)_{ij}y_{i}y_{j}\right\} 
 \nonumber\\
&\ \ \ \times \left\{ \sum_{|\beta|=k} \frac{1}{\beta!}\partial^\beta g(\bc) \left(\frac{\by}{\sqrt{n}}\right)^\beta +  \sum_{|\beta|=k+1} \frac{1}{\beta!}\partial^\beta g\left(\bc+\theta_3\frac{\by}{\sqrt{n}}\right) \left(\frac{\by}{\sqrt{n}}\right)^\beta \right\}d\by  \nonumber \\
&\ \ \ +n^{-\frac{d}{2}}\int_{B_{\delta\sqrt{n}}(0)} M_{n}(\by)
\exp\left\{ \frac{1}{2}\sum_{i=1}^{d}\sum_{j=1}^{d}\left(D^{2}h_{n}(\bc_{n})\right)_{ij}y_{i}y_{j}\right\} 
 \nonumber\\
&\ \ \ \times \left\{ \sum_{|\beta|=k} \frac{1}{\beta!}\partial^\beta g(\bc) \left(\frac{\by}{\sqrt{n}}\right)^\beta +  \sum_{|\beta|=k+1} \frac{1}{\beta!}\partial^\beta g\left(\bc+\theta_3\frac{\by}{\sqrt{n}}\right) \left(\frac{\by}{\sqrt{n}}\right)^\beta \right\}d\by  \nonumber \\
&\ \ \ +n^{-\frac{d}{2}-\frac{1}{2}}\int_{B_{\delta\sqrt{n}}(0)} \left(1+M_{n}(\by)\right) \left\{\sum_{|\alpha|=3} \frac{1}{\alpha!} \partial^\alpha h_{n} \left(\bc_{n}+\theta_0\frac{\by}{\sqrt{n}} -\theta_{0}\varepsilon_{n}\ba_{n}\right) \left(\by-\varepsilon_{n}\sqrt{n}\ba_{n}\right)^\alpha \right\} \nonumber \\
&\ \ \ \times \exp\left\{\frac{1}{2}\sum_{i=1}^{d}\sum_{j=1}^{d}\left(D^{2}h_{n}(\bc_{n})\right)_{ij}y_{i}y_{j}+\frac{\theta_5}{\sqrt{n}} \sum_{|\alpha|=3} \frac{1}{\alpha!} \partial^\alpha h_{n} \left(\bc_{n}+\theta_0\frac{\by}{\sqrt{n}} -\theta_{0}\varepsilon_{n}\ba_{n}\right) \left(\by-\varepsilon_{n}\sqrt{n}\ba_{n}\right)^\alpha \right\} \nonumber \\
&\ \ \ \times \left\{ \sum_{|\beta|=k} \frac{1}{\beta!}\partial^\beta g(\bc) \left(\frac{\by}{\sqrt{n}}\right)^\beta +  \sum_{|\beta|=k+1} \frac{1}{\beta!}\partial^\beta g\left(\bc+\theta_3\frac{\by}{\sqrt{n}}\right) \left(\frac{\by}{\sqrt{n}}\right)^\beta \right\} d\by \nonumber \\
&=:\left(J_n^{(2,1)}+J_n^{(2,2)}+J_n^{(2,3)}\right). \label{J1,1+J1,2}
\end{align}

Now, let us derive the leading term of $I_{n}e^{-nh_{n}(\bc_{n})}$ from $J_n^{(2,1)}$. 
We start with calculating the integral. By virtue of Lemma \ref{lem:gaussian}, we can see that the following fact holds: 
\begin{align}
&n^{-\frac{d}{2}} \int_{\R^d}\exp\left\{\frac{1}{2}\sum_{i=1}^{d}\sum_{j=1}^{d}\left(D^{2}h_{n}(\bc_{n})\right)_{ij}y_{i}y_{j}\right\} \sum_{|\beta|=k} \frac{1}{\beta!} \partial^{\beta} g(\bc) \left(\frac{\by}{\sqrt{n}}\right)^\beta d\by \nonumber \\
&=
n^{-\frac{d}{2}-\frac{k}{2}} \sum_{|\beta|=k} \frac{1}{\beta!} \partial^\beta g(\bc) \int_{\R^d} \exp\left\{\frac{1}{2}\sum_{i=1}^{d}\sum_{j=1}^{d}\left(D^{2}h_{n}(\bc_{n})\right)_{ij}y_{i}y_{j}\right\}   \by^\beta d\by \nonumber  \\
&=n^{-\frac{d}{2}-\frac{k}{2}}\hspace{-5mm} \sum_{\substack{|\beta|=k\\ \beta=(\beta_1,\dots, \beta_d) \in (2\Z_{\ge0})^d}} \hspace{-5mm}\frac{1}{\beta!} \partial^\beta g(\bc)\prod_{i=1}^d \left( \frac{|\lambda_{i, n}(\bc_{n})|}{2}\right)^{-\frac{\beta_i+1}{2}} \Gamma\left( \frac{\beta_i +1}{2}\right) \nonumber \\
&=n^{-\frac{d}{2}-\frac{k}{2}} \hspace{-5mm} \sum_{\substack{|\beta|=k\\ \beta=(\beta_1,\dots, \beta_d) \in (2\Z_{\ge0})^d}} \hspace{-5mm} \frac{1}{\beta_{1}!\cdots \beta_{d}!} \partial^\beta g(\bc)\prod_{i=1}^d \left( \frac{|\lambda_{i, n}(\bc_{n})|}{2}\right)^{-\frac{\beta_i+1}{2}} \frac{\sqrt{\pi}(\beta_{i}-1)!!}{2^{\frac{\beta_{i}}{2}}} \nonumber \\
&=n^{-\frac{d}{2}-\frac{k}{2}}\sqrt{\frac{(2\pi)^d}{|\det D^{2}h_{n}(\bc_{n})|}}\hspace{-5mm} \sum_{\substack{|\beta|=k\\ \beta=(\beta_1,\dots, \beta_d) \in (2\Z_{\ge0})^d}} \hspace{-5mm} \left(\partial^\beta g(\bc) \right)\prod_{i=1}^d \frac{|\lambda_{i, n}(\bc_{n})|^{-\frac{\beta_i}{2}}}{\beta_i !!}. \label{leadn-junbi}
\end{align}
Here, we shall consider an expansion of $1/\sqrt{\det D^{2}h_{n}(\bc_{n})}$ in the above. 
Now, noticing that if $n\in \mathbb{N}$ is sufficiently large, then one can choose $C_{n}\in \R$ appeared from \eqref{Cn-est}, such that $|C_{n}|<|\det D^{2}h(\bc)|/2$,
\begin{equation}\label{Cn-bound}
\text{i.e.,} \ \ \frac{1}{2}<1+\frac{C_{n}}{\det D^{2}h(\bc)}<\frac{3}{2}.
\end{equation}
Therefore, we obtain the following expansion of $1/\sqrt{\det D^{2}h_{n}(\bc_{n})}$: 
\begin{align}
\sqrt{\frac{1}{|\det D^{2}h_{n}(\bc_{n})|}}
&=\sqrt{\frac{1}{|\det D^{2}h(\bc)+C_{n}|}}
=\sqrt{\frac{1}{|\det D^{2}h(\bc)|}}\sqrt{\frac{1}{|1+\frac{C_{n}}{\det D^{2}h(\bc)}|}}  \nonumber \\
&=\sqrt{\frac{1}{|\det D^{2}h(\bc)|}}\sqrt{\frac{1}{1+\frac{C_{n}}{\det D^{2}h(\bc)}}} 
=\sqrt{\frac{1}{|\det D^{2}h(\bc)|}}\left(1+\frac{1-\sqrt{1+\frac{C_{n}}{\det D^{2}h(\bc)}}}{\sqrt{1+\frac{C_{n}}{\det D^{2}h(\bc)}}}\right) \nonumber \\
&=\sqrt{\frac{1}{|\det D^{2}h(\bc)|}}\left\{1+\frac{-\frac{C_{n}}{\det D^{2}h(\bc)}}{\sqrt{1+\frac{C_{n}}{\det D^{2}h(\bc)}}\left(1+\sqrt{1+\frac{C_{n}}{\det D^{2}h(\bc)}}\right)}\right\} \nonumber  \\
&=:\sqrt{\frac{1}{|\det D^{2}h(\bc)|}}\left(1+E_{n}\right). \label{det-expan}
\end{align}
Note that the following error bound for $E_{n}$ directly follows from \eqref{det-expan}, \eqref{Cn-bound}, \eqref{Cn-est} and \eqref{det-expan-prop}: 
\begin{equation}\label{En-est}
|E_{n}| \le  \frac{|C_{n}|}{|\det D^{2}h(\bc)|}\frac{1}{\frac{1}{\sqrt{2}}\left(1+\frac{1}{\sqrt{2}}\right)}=O\left(n^{-p}\right). 
\end{equation}
Next, let us give an expansion of $|\lambda_{i, n}(\bc_{n})|^{-\frac{\beta_{i}}{2}}$ in \eqref{leadn-junbi}. By a similar argument to that from \eqref{Cn-bound} to \eqref{En-est}, if we take $n\in \mathbb{N}$ is sufficiently large, the value $\Lambda_{i, n} = \lambda_{i,n}(\bc_n)-\lambda_{i}(\bc)$ can be controlled easily. Then, using Taylor's theorem and Proposition \ref{lem:eigen}, there exists $\theta_{6}=\theta_{6}(\bc, i, n)\in (0, 1)$ such that 
\begin{align}
&\prod_{i=1}^d \frac{|\lambda_{i, n}(\bc_{n})|^{-\frac{\beta_i}{2}}}{\beta_i !!}
=\prod_{i=1}^d \frac{|\lambda_{i}(\bc)+\Lambda_{i, n}|^{-\frac{\beta_i}{2}}}{\beta_i !!}
=\prod_{i=1}^d \frac{|\lambda_{i}(\bc)|^{-\frac{\beta_i}{2}}}{\beta_i !!}\left|1+\frac{\Lambda_{i, n}}{\lambda_{i}(\bc)}\right|^{-\frac{\beta_i}{2}}  \nonumber \\
&\quad =\prod_{i=1}^d \frac{|\lambda_{i}(\bc)|^{-\frac{\beta_i}{2}}}{\beta_i !!}\left(1+\frac{\Lambda_{i, n}}{\lambda_{i}(\bc)}\right)^{-\frac{\beta_i}{2}} 
=\prod_{i=1}^d \frac{|\lambda_{i}(\bc)|^{-\frac{\beta_i}{2}}}{\beta_i !!}\left\{1-\frac{\beta_{i}}{2}\left(1+\frac{\theta_{6}\Lambda_{i, n}}{\lambda_{i}(\bc)}\right)^{-\frac{\beta_{i}}{2}-1} \frac{\Lambda_{i, n}}{\lambda_{i}(\bc)}\right\} \nonumber \\
&\quad =:\prod_{i=1}^d \frac{|\lambda_{i}(\bc)|^{-\frac{\beta_i}{2}}}{\beta_i !!}(1-H_{i, n})  \ \ \text{and} \ \ H_{i, n}=O\left(n^{-p}\right). \label{DEF-Hin}
\end{align}
Therefore, it follows from the definition of $J_n^{(2,1)}$ in \eqref{J1,1+J1,2}, the formulas \eqref{leadn-junbi}, \eqref{det-expan} and \eqref{DEF-Hin} that 
\begin{align}
J_n^{(2,1)} 
&=n^{-\frac{d}{2}} \left(\int_{\R^d}-\int_{\R^d \setminus B_{\delta\sqrt{n}}(0)} \right) 
\exp\left\{ \frac{1}{2}\sum_{i=1}^{d}\sum_{j=1}^{d}\left(D^{2}h_{n}(\bc_{n})\right)_{ij}y_{i}y_{j}\right\} 
 \sum_{|\beta|=k} \frac{1}{\beta!} \partial^{\beta} g(\bc) \left(\frac{\by}{\sqrt{n}}\right)^\beta d\by \nonumber \\
&\ \ \ +n^{-\frac{d}{2}} \int_{B_{\delta\sqrt{n}}(0)}\exp\left\{ \frac{1}{2}\sum_{i=1}^{d}\sum_{j=1}^{d}\left(D^{2}h_{n}(\bc_{n})\right)_{ij}y_{i}y_{j}\right\}  \sum_{|\beta|=k+1} \frac{1}{\beta!} \partial^{\beta} g\left(\bc + \theta_3\frac{\by}{\sqrt{n}}\right) \left(\frac{\by}{\sqrt{n}}\right)^\beta d\by \nonumber\\
&=n^{-\frac{d}{2}-\frac{k}{2}}(1+E_{n})\sqrt{\frac{(2\pi)^d}{|\det D^{2}h(\bc)|}} \hspace{-5mm} \sum_{\substack{|\beta|=k\\ \beta=(\beta_1,\dots, \beta_d) \in (2\Z_{\ge0})^d}} \hspace{-5mm} \left(\partial^\beta g(\bc) \right)\prod_{i=1}^d \frac{|\lambda_i(\bc)|^{-\frac{\beta_i}{2}}}{\beta_i !!}(1-H_{i, n}) \nonumber \\
&\ \ \ - n^{-\frac{d}{2}-\frac{k}{2}} \int_{\R^d \setminus B_{\delta\sqrt{n}}(0)} \exp\left\{ \frac{1}{2}\sum_{i=1}^{d}\sum_{j=1}^{d}\left(D^{2}h_{n}(\bc_{n})\right)_{ij}y_{i}y_{j}\right\}  \sum_{|\beta|=k} \frac{1}{\beta!} \partial^{\beta} g(\bc) \by^\beta d\by \nonumber\\
&\ \ \ +n^{-\frac{d}{2}-\frac{k}{2}-\frac{1}{2}} \int_{B_{\delta\sqrt{n}}(0)}\exp\left\{ \frac{1}{2}\sum_{i=1}^{d}\sum_{j=1}^{d}\left(D^{2}h_{n}(\bc_{n})\right)_{ij}y_{i}y_{j}\right\} \sum_{|\beta|=k+1} \frac{1}{\beta!} \partial^{\beta} g\left(\bc + \theta_3\frac{\by}{\sqrt{n}}\right) \by^\beta d\by \nonumber\\
&=:J_n^{(2,1,1)}-J_n^{(2,1,2)}+J_n^{(2,1,3)}. \label{leading-term}
\end{align}
Then, applying \eqref{En-est} and \eqref{DEF-Hin} to $J_n^{(2,1,1)}$, we can easily have
\begin{equation}\label{lead-ex}
J_n^{(2,1,1)}=n^{-\frac{d}{2}-\frac{k}{2}}\sqrt{\frac{(2\pi)^d}{|\det D^{2}h(\bc)|}} \hspace{-5mm} \sum_{\substack{|\beta|=k\\ \beta=(\beta_1,\dots, \beta_d) \in (2\Z_{\ge0})^d}} \hspace{-5mm} \left(\partial^\beta g(\bc) \right)\prod_{i=1}^d \frac{|\lambda_i(\bc)|^{-\frac{\beta_i}{2}}}{\beta_i !!}
+O\left(n^{-\frac{d}{2}-\frac{k}{2}-p}\right). 
\end{equation}
Thus, we were able to obtain the leading term of $I_{n}e^{-nh_{n}(\bc_{n})}$. 

On the other hand, $J_n^{(2,1,2)}$, $J_n^{(2,1,3)}$, $J_n^{(2, 2)}$ and $J_n^{(2, 3)}$ are error terms. 
Actually, we can evaluate them from direct calculations. First, we deal with $J_n^{(2,1,2)}$ and $J_n^{(2,1,3)}$. It follows from the condition {\bf (vi)} that 
\begin{align}
\left|J_n^{(2,1,2)}\right| 
&\le n^{-\frac{d}{2}-\frac{k}{2}} \sup_{\R^d\setminus B_{\delta\sqrt{n}}(0)} \exp\left\{ \frac{1}{4}\sum_{i=1}^{d}\sum_{j=1}^{d}\left(D^{2}h_{n}(\bc_{n})\right)_{ij}y_{i}y_{j}\right\}  \nonumber \\
&\ \ \ \times \hspace{-5mm}\sum_{\substack{|\beta|=k\\ \beta=(\beta_1,\dots, \beta_d) \in \Z_{\ge0}^d}}\hspace{-5mm} \frac{|\partial^\beta g(\bc)|}{\beta!} \int_{\R^d\setminus B_{\delta\sqrt{n}}(0)}\exp\left\{ \frac{1}{4}\sum_{i=1}^{d}\sum_{j=1}^{d}\left(D^{2}h_{n}(\bc_{n})\right)_{ij}y_{i}y_{j}\right\}|y_{1}|^{\beta_{1}}\cdots |y_{d}|^{\beta_{d}} d\by \nonumber\\
&\le n^{-\frac{d}{2}-\frac{k}{2}} \sup_{\R^d\setminus B_{\delta\sqrt{n}}(0)} \exp\left(-\frac{1}{4}\sum_{i=1}^{d} |\lambda_{i, n}(\bc_{n})|^2 y_i^2\right)  \nonumber \\
&\ \ \ \times \hspace{-5mm} \sum_{\substack{|\beta|=k\\ \beta=(\beta_1,\dots, \beta_d) \in \Z_{\ge0}^d}} \hspace{-5mm} \frac{|\partial^\beta g(\bc)|}{\beta!}  \int_{\R^d}\exp\left(-\frac{1}{4}\sum_{i=1}^{d} |\lambda_{i, n}(\bc_{n})|^2 y_i^2\right)|y_{1}|^{\beta_{1}}\cdots |y_{d}|^{\beta_{d}} d\by \nonumber \\
&= n^{-\frac{d}{2}-\frac{k}{2}} \exp\left(-\frac{\delta^{2} n}{4}\sum_{i=1}^{d} |\lambda_{i, n}(\bc_{n})|^2 \right) \hspace{-5mm} \sum_{\substack{|\beta|=k\\ \beta=(\beta_1,\dots, \beta_d) \in \Z_{\ge0}^d}} \hspace{-5mm} \frac{|\partial^\beta g(\bc)|}{\beta!} \prod_{i=1}^{d}\left(\frac{|\lambda_{i, n}(\bc_{n})|}{2}\right)^{-(\beta_{i}+1)}\Gamma\left(\frac{\beta_{i}+1}{2}\right) \nonumber \\
&\le n^{-\frac{d}{2}-\frac{k}{2}} \exp\left(-\frac{C_{\dag} \delta^{2} n}{4}\right) \hspace{-5mm}\sum_{\substack{|\beta|=k\\ \beta=(\beta_1,\dots, \beta_d) \in \Z_{\ge0}^d}}\hspace{-5mm} \frac{|\partial^\beta g(\bc)|}{\beta!} \prod_{i=1}^{d}\left(\frac{C_{\dag}}{2}\right)^{-(\beta_{i}+1)}\Gamma\left(\frac{\beta_{i}+1}{2}\right) \nonumber \\
&=O\left(n^{-\frac{d}{2}-\frac{k}{2}} \exp\left(-\frac{C_{\dag} \delta^{2} n}{4}\right)\right) \label{est-J111}
\end{align}
and 
\begin{align}
\left|J_n^{(2,1,3)}\right|
&\le n^{-\frac{d}{2}-\frac{k}{2}-\frac{1}{2}} \hspace{-5mm} \sum_{\substack{|\beta|=k+1\\ \beta=(\beta_1,\dots, \beta_d) \in \Z_{\ge0}^d}} \hspace{-5mm} \frac{1}{\beta!} \left\|\partial^\beta g\right\|_{C^{0}(\overline{\Omega})} 
\int_{\R^d}\exp\left(-\frac{1}{2}\sum_{i=1}^{d} |\lambda_{i, n}(\bc_{n})|^2 y_i^2\right)|y_{1}|^{\beta_{1}}\cdots |y_{d}|^{\beta_{d}} d\by  \nonumber \\
&=n^{-\frac{d}{2}-\frac{k}{2}-\frac{1}{2}} \hspace{-5mm} \sum_{\substack{|\beta|=k+1\\ \beta=(\beta_1,\dots, \beta_d) \in \Z_{\ge0}^d}} \hspace{-5mm} \frac{1}{\beta!} \left\|\partial^\beta g\right\|_{C^{0}(\overline{\Omega})} \prod_{i=1}^{d}\left(\frac{|\lambda_{i, n}(\bc_{n})|}{2}\right)^{-(\beta_{i}+1)}\Gamma\left(\frac{\beta_{i}+1}{2}\right)  \nonumber \\
&\le n^{-\frac{d}{2}-\frac{k}{2}-\frac{1}{2}} \hspace{-5mm} \sum_{\substack{|\beta|=k+1\\ \beta=(\beta_1,\dots, \beta_d) \in \Z_{\ge0}^d}} \hspace{-5mm} \frac{1}{\beta!} \left\|\partial^\beta g\right\|_{C^{0}(\overline{\Omega})} \prod_{i=1}^{d}\left(\frac{C_{\dag}}{2}\right)^{-(\beta_{i}+1)}\Gamma\left(\frac{\beta_{i}+1}{2}\right) \nonumber \\
&= O\left( n^{-\frac{d}{2}-\frac{k}{2}-\frac{1}{2}}\right). \label{est-J112}
\end{align}
Also, by virtue of \eqref{Mn-est}, we analogously obtain the estimate for $J_n^{(2,2)}$ like \eqref{est-J111} and \eqref{est-J112} as follows: 
\begin{align}
\left|J_n^{(2,2)}\right|
&\le n^{-\frac{d}{2}}\left\|M_{n}\right\|_{C^{0}(\overline{\Omega})}\biggl\{n^{-\frac{k}{2}}\hspace{-5mm}\sum_{\substack{|\beta|=k\\ \beta=(\beta_1,\dots, \beta_d) \in \Z_{\ge0}^d}} \hspace{-5mm} \frac{1}{\beta!}\left|\partial^{\beta}g(\bc)\right|\int_{\R^d}\exp\left(-\frac{1}{2}\sum_{i=1}^{d} |\lambda_{i, n}(\bc_{n})|^2 y_i^2\right)|y_{1}|^{\beta_{1}}\cdots |y_{d}|^{\beta_{d}} d\by  \nonumber \\
&\ \ \ +n^{-\frac{k}{2}-\frac{1}{2}}\hspace{-5mm}\sum_{\substack{|\beta|=k+1\\ \beta=(\beta_1,\dots, \beta_d) \in \Z_{\ge0}^d}}\hspace{-5mm}\frac{1}{\beta!}\left\|\partial^{\beta}g\right\|_{C^{0}(\overline{\Omega})}\int_{\R^d}\exp\left(-\frac{1}{2}\sum_{i=1}^{d} |\lambda_{i, n}(\bc_{n})|^2 y_i^2\right)|y_{1}|^{\beta_{1}}\cdots |y_{d}|^{\beta_{d}} d\by\biggl\} \nonumber\\
&\le n^{-\frac{d}{2}}\left\|M_{n}\right\|_{C^{0}(\overline{\Omega})}\biggl\{n^{-\frac{k}{2}}\hspace{-5mm}\sum_{\substack{|\beta|=k\\ \beta=(\beta_1,\dots, \beta_d) \in \Z_{\ge0}^d}}\hspace{-5mm}\frac{1}{\beta!}\left|\partial^{\beta}g(\bc)\right|\prod_{i=1}^{d}\left(\frac{C_{\dag}}{2}\right)^{-(\beta_{i}+1)}\Gamma\left(\frac{\beta_{i}+1}{2}\right)  \nonumber \\
&\ \ \ +n^{-\frac{k}{2}-\frac{1}{2}}\hspace{-5mm}\sum_{\substack{|\beta|=k+1\\ \beta=(\beta_1,\dots, \beta_d) \in \Z_{\ge0}^d}}\hspace{-5mm}\frac{1}{\beta!}\left\|\partial^{\beta}g\right\|_{C^{0}(\overline{\Omega})}\prod_{i=1}^{d}\left(\frac{C_{\dag}}{2}\right)^{-(\beta_{i}+1)}\Gamma\left(\frac{\beta_{i}+1}{2}\right)\biggl\} \nonumber\\
&=O\left(n^{-\frac{d}{2}-\frac{k}{2}-(p-1)}\right)+O\left(n^{-\frac{d}{2}-\frac{k}{2}-\frac{1}{2}-(p-1)}\right). \label{J22-est}
\end{align}

In the rest of this proof, we shall evaluate the other remainder term $J_n^{(2,3)}$ in \eqref{J1,1+J1,2}. 
In order to do that, let us rewrite $J_n^{(2,3)}$ as follows: 
\begin{align}
&J_n^{(2,3)}
= n^{-\frac{d}{2}-\frac{k}{2}-\frac{1}{2}} \sum_{|\alpha|=3} \sum_{|\beta|=k} \frac{1}{\alpha! \beta!}   \int_{B_{\delta\sqrt{n}}(0)} \left(1+M_{n}(\by)\right)\partial^\alpha h_{n} \left(\bc_{n} + \theta_0 \frac{\by}{\sqrt{n}}-\theta_{0}\varepsilon_{n}\ba_{n}\right) \partial^\beta g(\bc) \nonumber \\
&\ \ \ \times \exp\left\{\frac{1}{2}\sum_{i=1}^{d}\sum_{j=1}^{d}\left(D^{2}h_{n}(\bc_{n})\right)_{ij}y_{i}y_{j}+\frac{\theta_5}{\sqrt{n}} \sum_{|\alpha|=3} \frac{1}{\alpha!} \partial^\alpha h_{n} \left(\bc_{n}+\theta_0\frac{\by}{\sqrt{n}} -\theta_{0}\varepsilon_{n}\ba_{n}\right) \left(\by-\varepsilon_{n}\sqrt{n}\ba_{n}\right)^\alpha \right\} \nonumber\\
&\ \ \ \times \left(\by-\varepsilon_{n}\sqrt{n}\ba_{n}\right)^{\alpha}\by^{\beta}d\by \nonumber\\
&\ \ \ +n^{-\frac{d}{2}-\frac{k}{2}-1} \sum_{|\alpha|=3}\sum_{|\beta|=k+1} \frac{1}{\alpha!\beta!} \int_{B_{\delta\sqrt{n}}(0)}\left(1+M_{n}(\by)\right)\partial^\alpha h_{n} \left(\bc_{n} + \theta_0 \frac{\by}{\sqrt{n}}-\theta_{0}\varepsilon_{n}\ba_{n}\right) \partial^\beta g \left(\bc + \theta_3 \frac{\by}{\sqrt{n}}\right) \nonumber\\
&\ \ \ \times \exp\left\{\frac{1}{2}\sum_{i=1}^{d}\sum_{j=1}^{d}\left(D^{2}h_{n}(\bc_{n})\right)_{ij}y_{i}y_{j}+\frac{\theta_5}{\sqrt{n}} \sum_{|\alpha|=3} \frac{1}{\alpha!} \partial^\alpha h_{n} \left(\bc_{n}+\theta_0\frac{\by}{\sqrt{n}} -\theta_{0}\varepsilon_{n}\ba_{n}\right) \left(\by-\varepsilon_{n}\sqrt{n}\ba_{n}\right)^\alpha \right\} \nonumber\\
&\ \ \ \times \left(\by-\varepsilon_{n}\sqrt{n}\ba_{n}\right)^{\alpha}\by^{\beta}d\by. \label{J23-junbi}
\end{align}
In addition, we need to prepare an estimate for the exponential function in the above. 
In order to do that, we would like to analyze $\left(\by-\varepsilon_{n}\sqrt{n}\ba_{n}\right)^\alpha$ for $|\alpha|=3$, by using a similar argument to \eqref{L-est}. It can be divided into the following three cases:  

\smallskip
\noindent
{\bf (1)} Three of components are 1 and the others are 0, i.e., $\alpha=(0, \cdots, 1, \cdots, 1, \cdots, 1, \cdots, 0)$. \\[.3em]
\noindent
{\bf (2)} One of the components is 2, the other is 1, and the rest are 0, i.e., $\alpha=(0, \cdots, 2, \cdots, 1, \cdots, 0)$. \\[.3em]
\noindent
{\bf (3)} One of the components is 3, and the rest are 0, i.e., $\alpha=(0, \cdots, 3, \cdots, 0)$. 

\smallskip
\noindent
In any of the above cases, analogously as \eqref{L-est}, we can obtain the following result: 
\begin{equation}\label{Q-est}
\left(\by-\varepsilon_{n}\sqrt{n}\ba_{n}\right)^\alpha=\by^{\alpha}+ Q_{n}(\by, \alpha) \quad \text{with} \quad Q_{n}(\by, \alpha)=O\left(n^{-p+\frac{3}{2}}\right), \quad \text{for} \quad |\alpha|=3. 
\end{equation}
By virtue of this result, from the condition {\bf (v)} and {\bf (vi)}, we are able to see that the following estimate holds for all $\by \in B_{\delta\sqrt{n}}(0)$: 
\begin{align}
&\exp\left\{\frac{1}{2}\sum_{i=1}^{d}\sum_{j=1}^{d}\left(D^{2}h_{n}(\bc_{n})\right)_{ij}y_{i}y_{j}+\frac{\theta_5}{\sqrt{n}} \sum_{|\alpha|=3} \frac{1}{\alpha!} \partial^\alpha h_{n} \left(\bc_{n}+\theta_0\frac{\by}{\sqrt{n}} -\theta_{0}\varepsilon_{n}\ba_{n}\right) \left(\by-\varepsilon_{n}\sqrt{n}\ba_{n}\right)^\alpha \right\} \nonumber \\
&=\exp\left\{-\frac{1}{2}\sum_{i=1}^{d}\left|\lambda_{i, n}(\bc_{n})\right|^{2}y_{i}^{2}+\frac{\theta_5}{\sqrt{n}} \sum_{|\alpha|=3} \frac{1}{\alpha!} \partial^\alpha h_{n} \left(\bc_{n}+\theta_0\frac{\by}{\sqrt{n}} -\theta_{0}\varepsilon_{n}\ba_{n}\right) \left(\by^{\alpha}+Q_{n}(\by, \alpha)\right) \right\} \nonumber\\
&\le \exp\left\{ -\frac{C_{\dag}}{2} |\by|^{2}  +  \delta|\by|^{2}\sum_{|\alpha|=3} \frac{1}{\alpha!}\left\|\partial^\alpha h_{n}\right\|_{C^{0}(\overline{\Omega})}+\sum_{|\alpha|=3} \frac{1}{\alpha!}\left\|\partial^\alpha h_{n}\right\|_{C^{0}(\overline{\Omega})}\frac{\left\|Q_{n}(\cdot, \alpha)\right\|_{C^{0}(\overline{\Omega})}}{\sqrt{n}} \right\} \nonumber\\
&= \exp\left\{ -\left(\frac{C_{\dag}}{2}-\delta \sum_{|\alpha|=3} \frac{1}{\alpha!}\left\|\partial^\alpha h_{n}\right\|_{C^{0}(\overline{\Omega})} \right)|\by|^{2}\right\} \exp\left(\sum_{|\alpha|=3} \frac{1}{\alpha!}\left\|\partial^\alpha h_{n}\right\|_{C^{0}(\overline{\Omega})}\frac{\left\|Q_{n}(\cdot, \alpha)\right\|_{C^{0}(\overline{\Omega})}}{\sqrt{n}}  \right) \nonumber\\
&\le C_{0}\exp\left(-B|\by|^{2}\right), \label{exp-est-new}
\end{align}
where the constant $B\in \R$ and $C_{0}>0$ are defined by
\[
B:=\frac{C_{\dag}}{2}-\delta \sum_{|\alpha|=3} \frac{1}{\alpha!}\sup_{n\ge N_{1}}\left\|\partial^\alpha h_{n}\right\|_{C^{0}(\overline{\Omega})}, \quad
C_{0}:=\exp\left\{\sup_{n\ge N_{1}}\left(\sum_{|\alpha|=3} \frac{1}{\alpha!}\left\|\partial^\alpha h_{n}\right\|_{C^{0}(\overline{\Omega})}\frac{\left\|Q_{n}(\cdot, \alpha)\right\|_{C^{0}(\overline{\Omega})}}{\sqrt{n}} \right)\right\}. 
\]
Noticing that if $\delta>0$ is small enough, then we can consider $B>0$. 
Also, we emphasize that $C_{0}<\infty$, since $n^{-\frac{1}{2}}\left\|Q_{n}(\cdot, \alpha)\right\|_{C^{0}(\overline{\Omega})}=O\left(n^{-(p-1)}\right)$ follows from \eqref{Q-est}. 
Therefore, it follows from \eqref{J23-junbi}, \eqref{Q-est}, \eqref{exp-est-new}, \eqref{Mn-est} and the condition {\bf (v)} that 
\begin{align}
\left|J_n^{(2,3)}\right| 
&\le n^{-\frac{d}{2}-\frac{k}{2}-\frac{1}{2}} \sum_{|\alpha|=3}\sum_{|\beta|=k} \frac{1}{\alpha!\beta!} \sup_{n\ge N_{1}}\left\|\partial^\alpha h_{n}\right\|_{C^{0}(\overline{\Omega})}\left|\partial^\beta g(\bc)\right|\nonumber \\
&\ \ \ \times C_{0}\int_{\R^d} \left(1+\left\|M_{n}\right\|_{C^{0}(\overline{\Omega})}\right)\exp\left(-B|\by|^{2}\right) \left(|\by|^{|\alpha|}+\left\|Q_{n}(\cdot, \alpha)\right\|_{C^{0}(\overline{\Omega})}\right)|\by|^{|\beta|}d\by \nonumber \\
&\ \ \ +n^{-\frac{d}{2}-\frac{k}{2}-1} \sum_{|\alpha|=3}\sum_{|\beta|=k+1} \frac{1}{\alpha!\beta!} \sup_{n\ge N_{1}}\left\|\partial^\alpha h_{n}\right\|_{C^{0}(\overline{\Omega})}\left\|\partial^\beta g\right\|_{C^{0}(\overline{\Omega})}\nonumber \\
&\ \ \ \times C_{0}\int_{\R^d} \left(1+\left\|M_{n}\right\|_{C^{0}(\overline{\Omega})}\right)\exp\left(-B|\by|^{2}\right) \left(|\by|^{|\alpha|}+\left\|Q_{n}(\cdot, \alpha)\right\|_{C^{0}(\overline{\Omega})}\right)|\by|^{|\beta|}d\by \nonumber \\
&= n^{-\frac{d}{2}-\frac{k}{2}-\frac{1}{2}} \sum_{|\alpha|=3}\sum_{|\beta|=k} \frac{1}{\alpha!\beta!} \sup_{n\ge N_{1}}\left\|\partial^\alpha h_{n}\right\|_{C^{0}(\overline{\Omega})}\left|\partial^\beta g(\bc)\right|\nonumber \\
&\ \ \ \times \biggl\{C_{0}\left(1+\left\|M_{n}\right\|_{C^{0}(\overline{\Omega})}\right)\int_{\R^d} \exp\left(-B|\by|^{2}\right) |\by|^{|\alpha|+|\beta|}d\by \nonumber \\
&\ \ \ \ \ \ \ \ \ +C_{0}\left(1+\left\|M_{n}\right\|_{C^{0}(\overline{\Omega})}\right)\left\|Q_{n}(\cdot, \alpha)\right\|_{C^{0}(\overline{\Omega})}\int_{\R^d} \exp\left(-B|\by|^{2}\right) |\by|^{|\beta|}d\by \biggl\} \nonumber \\
&\ \ \ +n^{-\frac{d}{2}-\frac{k}{2}-1} \sum_{|\alpha|=3}\sum_{|\beta|=k+1} \frac{1}{\alpha!\beta!} \sup_{n\ge N_{1}}\left\|\partial^\alpha h_{n}\right\|_{C^{0}(\overline{\Omega})}\left\|\partial^\beta g\right\|_{C^{0}(\overline{\Omega})}\nonumber \\
&\ \ \ \times \biggl\{C_{0}\left(1+\left\|M_{n}\right\|_{C^{0}(\overline{\Omega})}\right)\int_{\R^d} \exp\left(-B|\by|^{2}\right) |\by|^{|\alpha|+|\beta|}d\by \nonumber \\
&\ \ \ \ \ \ \ \ \ +C_{0}\left(1+\left\|M_{n}\right\|_{C^{0}(\overline{\Omega})}\right)\left\|Q_{n}(\cdot, \alpha)\right\|_{C^{0}(\overline{\Omega})}\int_{\R^d} \exp\left(-B|\by|^{2}\right) |\by|^{|\beta|}d\by \biggl\} \nonumber \\
&=O\left( n^{-\frac{d}{2}-\frac{k}{2}-\frac{1}{2}}\right)+O\left( n^{-\frac{d}{2}-\frac{k}{2}-(p-1)}\right)
+O\left( n^{-\frac{d}{2}-\frac{k}{2}-1}\right)+O\left( n^{-\frac{d}{2}-\frac{k}{2}-\left(p-\frac{1}{2}\right)}\right). \label{est-J12}
\end{align}

Eventually, summarizing up \eqref{J1+J2}, \eqref{est-J2}, \eqref{J1,1+J1,2}, \eqref{leading-term}, \eqref{lead-ex}, \eqref{est-J111},   \eqref{est-J112}, \eqref{J22-est} and \eqref{est-J12}, and comparing the decay orders obtained from each estimate, we arrive at the following result: 
\begin{align}
I_{n}e^{-nh_{n}(\bc_{n})}
&=n^{-\frac{d}{2}-\frac{k}{2}}\sqrt{\frac{(2\pi)^d}{|\det D^{2}h(\bc)|}} \hspace{-5mm}\sum_{\substack{|\beta|=k\\ \beta=(\beta_1,\dots, \beta_d) \in (2\Z_{\ge0})^d}} \hspace{-5mm} \left(\partial^\beta g(\bc) \right)\prod_{i=1}^d \frac{|\lambda_i(\bc)|^{-\frac{\beta_i}{2}}}{\beta_i !!} \nonumber \\
&\ \ \ \ +O\left( n^{-\frac{d}{2}-\frac{k}{2}-\frac{1}{2}}\right)+O\left( n^{-\frac{d}{2}-\frac{k}{2}-(p-1)}\right). \label{final-1}
\end{align}
Finally, we shall give an expansion of $e^{nh_{n}(\bc_{n})}$. It follows from \eqref{h-junbi-2} and the definition of $h_{n}$ that 
\begin{align}\label{finial-junbi-1}
e^{nh_{n}(\bc_{n})}=e^{nh(\bc_{n})+n\varepsilon_{n}\sigma(\bc_{n})}
=e^{nh(\bc)}\cdot e^{n\nabla h(\bc+\theta_{1}(\bc_{n}-\bc))(\bc_{n}-\bc)} \cdot e^{n\varepsilon_{n}\sigma(\bc_{n})}. 
\end{align}
By similar arguments as before, from Lemma \ref{lem:cn-c} and the condition {\bf (i)}, we can easily see that 
\[
\left|e^{n\nabla h(\bc+\theta_{1}(\bc_{n}-\bc))(\bc_{n}-\bc)} -1\right|
\le n\left|\bc_{n}-\bc \right|\left\|\nabla h\right\|_{C^{0}(\overline{\Omega})}\exp \left(n\left|\bc_{n}-\bc \right|\left\|\nabla h\right\|_{C^{0}(\overline{\Omega})}\right)
=O\left(n^{-(p-1)}\right)
\]
and
\[
\left|e^{n\varepsilon_{n}\sigma(\bc_{n})}-1\right|
\le n\varepsilon_{n}\left\|\sigma\right\|_{C^{0}(\overline{\Omega})}\exp \left(n\varepsilon_{n}\left\|\sigma\right\|_{C^{0}(\overline{\Omega})}\right)
=O\left(n^{-(p-1)}\right). 
\]
These estimates mean that 
\begin{align}\label{finial-junbi-2}
e^{n\nabla h(\bc+\theta_{1}(\bc_{n}-\bc))(\bc_{n}-\bc)}=1+O\left(n^{-(p-1)}\right), \quad 
e^{n\varepsilon_{n}\sigma(\bc_{n})}=1+O\left(n^{-(p-1)}\right). 
\end{align}
Then, \eqref{finial-junbi-1} and \eqref{finial-junbi-2} lead the following expansion:   
\begin{align}
e^{nh_{n}(\bc_{n})}=e^{nh(\bc)}\left(1+O\left(n^{-(p-1)}\right)\right). \label{final-2}
\end{align}
Therefore, summing up \eqref{final-1} and \eqref{final-2}, and rearranging the equation, we are able to conclude that the desired asymptotic formula \eqref{asymp} is true. This completes the proof of Theorem \ref{thm:main}. 
\qed

\subsection*{Acknowledgement}

This study is supported by Grant-in-Aid for Young Scientists Research No.22K13925 and No.22K13939, Japan Society for the Promotion of Science.



\bigskip

\hspace{-6mm}{\bf Ikki Fukuda}\\
Division of Mathematics and Physics, \\
Faculty of Engineering, Shinshu University, \\
4-17-1, Wakasato, Nagano, 380-8553, JAPAN\\
E-mail: i\_fukuda@shinshu-u.ac.jp\\

\hspace{-6mm}{\bf Yoshiki Kagaya}\\
Department of Mathematical Science, \\
Graduate School of Science and Technology, Shinshu University, \\
3-1-1, Asahi, Matsumoto, 390-8621, JAPAN\\
E-mail: yoshiki.kagaya@gmail.com\\

\hspace{-6mm}{\bf Yuki Ueda}\\
Department of Mathematics, \\
Faculty of Education, Hokkaido University of Education, \\
9 Hokumon-cho, Asahikawa, Hokkaido, 070-8621, JAPAN\\
E-mail: ueda.yuki@a.hokkyodai.ac.jp

\end{document}